\documentclass[oneside, 10pt]{article}
\usepackage{amsmath, amsthm, amsfonts, amssymb, mathrsfs}
\usepackage[all]{xy}
\usepackage{color}
\usepackage{authblk}
\usepackage{afterpage}
\usepackage{array}
\usepackage{color}
\usepackage{hyperref}
\usepackage{float}
\usepackage[capitalise]{cleveref}
\usepackage{enumitem}
\usepackage{caption}
\usepackage[backend=biber,style=numeric,sorting=ynt,backref=true,backend=bibtex8]{biblatex}
\usepackage[a4paper,
            bindingoffset=0.2in,
            left=1.5in,
            right=1.5in,
            top=1.5in,
            bottom=1.5in,
            footskip=.5in]{geometry}

\theoremstyle{plain}
\newtheorem{theorem}{Theorem}[section]
\newtheorem{proposition}[theorem]{Proposition}
\newtheorem{corollary}[theorem]{Corollary}
\newtheorem{lemma}[theorem]{Lemma}
\theoremstyle{definition}
\newtheorem{definition}[theorem]{Definition}
\newtheorem{example}[theorem]{Example}
\newtheorem{remark}[theorem]{Remark}

\theoremstyle{plain}
\newtheorem{theorem*}{Theorem}
\newtheorem*{proposition*}{Proposition}
\newtheorem{corollary*}[theorem*]{Corollary}
\newtheorem*{lemma*}{Lemma}
\theoremstyle{definition}
\newtheorem*{definition*}{Definition}
\newtheorem*{remark*}{Remark}
\newtheorem*{notation*}{Notation}

\counterwithin{table}{section}

\newcommand{\CC}{\mathbb{C}}
\newcommand{\PP}{\mathbb{P}}
\newcommand{\ZZ}{\mathbb{Z}}
\newcommand{\NN}{\mathbb{N}}

\newcommand{\OO}{\mathcal{O}}
\newcommand{\FF}{\mathcal{F}}
\newcommand{\EE}{\mathcal{E}}
\newcommand{\LL}{\mathcal{L}}

\newcommand{\C}{\mathcal{C}}
\newcommand{\T}{\mathcal{T}}

\newcommand{\m}{\mathfrak{m}}
\newcommand{\g}{\mathfrak{g}}
\newcommand{\X}{\mathfrak{X}}
\newcommand{\aff}{\mathfrak{aff}}

\newcommand\ii{{\imath}}

\newcommand{\VV}{\mathscr{V}}
\newcommand{\pp}{\mathfrak{p}}

\DeclareMathAlphabet{\mathpzc}{OT1}{pzc}{m}{it}

\newcommand{\spec}{\mathrm{Spec}}
\newcommand{\codim}{\mathrm{codim}}
\newcommand{\depth}{\mathrm{depth}}

\newcommand{\sing}{\mathrm{Sing}}

\newcommand{\ov}{\overline}

\providecommand{\keywords}[1]
{
  \textbf{\textit{Keywords:}} #1
}

\title{Stability of pullbacks of foliations on weighted projective spaces}
\date{}
\author[1]{Javier Gargiulo Acea \thanks{Email: jngargiulo@gmail.com.  The author was fully supported by IMPA and CNPq, Brazil.}}
\affil[1]{\small Departamento de Matem\'atica Aplicada, Universidade Federal Fluminense, Brazil.}
\author[2]{Ariel Molinuevo \thanks{Email: arielmolinuevo@gmail.com.  The author was fully supported by UFRJ, Brazil.}}
\affil[2]{\small Instituto de Matematica, Universidade Federal do Rio de Janeiro, Brazil.}
\author[3]{Federico Quallbrunn \thanks{Email: fquallb@dm.uba.ar. The author was fully supported by CONICET, Argentina.}}
\affil[3]{\small Departamento de Matem\'atica, Universidad CAECE, Argentina.}
\author[4]{Sebastián Velazquez \thanks{Email: sebastian.velazquez@kcl.ac.uk. The author was fully supported by CNPq, Brazil, CONICET, Argentina  and EPSRC, United Kingdom.}}
\affil[4]{\small Department of Mathematics, King's College London, United Kingdom.}

\begin{document}

\maketitle
\begin{abstract} We show a stability-type theorem for foliations on projective spaces which arise as pullbacks of foliations with a split tangent sheaf on weighted projective spaces. As a consequence, we will be able to construct many irreducible components of the moduli spaces of foliations of codimension 1 on $\PP^n$, most of them being previously unknown. This result also provides an alternative and unified proof for the stability of other families of foliations.\\
\end{abstract}

\tableofcontents

\bigskip

{\small \noindent \keywords{Singular projective foliations - Moduli spaces - Weighted projective spaces - Rational pullbacks.}}

\medskip


\newpage
\section{Introduction}

In the present article we address the problem of classifying foliations on projective spaces by means of studying the geometry of the space of codimension $1$ foliations of a given degree $d\in\ZZ$ on  $\PP^n_\CC$. This is the quasi-projective scheme
\[	\FF^1(\PP^n,d)=\{[\omega]\in \PP H^0(\PP^n,{\Omega}^1_{\PP^n}(d)): \omega \wedge d \omega =0  \,\mbox{ and } \, \codim(\sing(\omega))\geq 2   \},
\]
where $\sing(\omega)=\{x\in\PP^n: \omega(x)=0\}$.
In \cite{jou} is proven that for $n\geq 3$ and $d=2$ there is only one component of this space while for $n\geq 3$ and $d=3$ there are two components. In \cite{celn} a complete classification is found for the case $n\geq 3$ and $d=4$, which consists of six components. Since then, progress has been made through several articles in identifying new components of the space of foliations for different values of $n$ and $d$.

Foliations of pullback-type have played a central role in the classification of codimension one foliations in projective spaces. For instance, Brunella's conjecture states that every foliation on $\PP^n$ ($n\geq 3$) is either the pullback under a rational map of a foliation on a surface or admits invariant algebraic hypersurfaces. Taking this into account, it is only natural that the study of the geometry of the spaces $\FF^1(\PP^n,d)$ around foliations of pullback-type should lead to significant progress towards this issue.

This article started as an effort to understand and extend the main result of \cite{CLE}, which is perhaps the most important work tackling the above matter. There, the authors show that the closure of the set of foliations $\omega$ which can be written as $\omega = F^* \alpha$, where $\alpha$ is a foliation of degree $e$ in $\PP^2$ and $F : \PP^n \dashrightarrow \PP^2$ is a generic rational map defined by polynomials $(F_0,F_1,F_2)$ each of degree $v$, is an irreducible component of the space $\FF^1(\PP^n, v.e)$.
Since then, it remained an open problem in the area whether there exist other irreducible components whose generic points are pullbacks of foliations on other surfaces or more generally other higher dimensional varieties $X$. There has been very little progress towards this  problem, even for the simplest case where $X$ is a weighted projective plane. In this article we give a positive answer to this question when $X$ is a weighted projective space of arbitrary dimension. As a consequence, we are able to construct many new families of irreducible components of the spaces of foliations (see Table 1.1 below), while at the same time giving a unified proof for constructions appearing in probably the most influential works addressing this topic.

Another important work regarding this matter is \cite{fj}, which deals with \emph{split foliations}. We say a foliation $\omega$ is of \emph{split type} if its corresponding tangent sheaf, $\T_{\omega}(U):= \{\X\in \Gamma(U, \T_{\PP^n})$ such that $ i_\X\omega=0\}$ splits as a sum of line bundles $\T_\omega\cong \bigoplus_i \mathcal{L}_i$. There it is shown that under some hypotheses on its singularities, every deformation of a split foliation is again split. As a consequence, the authors are able to deduce that for every irreducible component $\C\subseteq \FF^1(\PP^{m},d)$ whose generic element $\omega$ defines a split foliation with $\codim (\sing(d\omega))\geq 3$ there exists an irreducible component $\C'\subseteq \FF^1(\PP^{n+m},d)$ whose generic element is a pullback of an element in $\C$ under a linear projection $\PP^{n+m}\dashrightarrow \PP^m$.

 In this paper we find components whose general member is a foliation of the form $F^*\alpha$ with $F:\PP^n\dashrightarrow\PP^m(\overline{e})$ a dominant rational map from a projective space to a weighted projective space with weights $\overline{e}=(e_0,\dots,e_m)$ and $\alpha$ is a split foliation. See bellow for full statements. In this way we obtain in \cref{sectionApp} an infinite number of previously unknown components, and generalize the aforementioned result of \cite{fj} and the construction of the \textit{exceptional components} made in \cite{celn} in the cases  $n\geq m+2$ and $n\geq 4$ respectively.

 Of course, our approach gives a new proof of the result of \cite{CLE} in the case $n\geq 4$ and shows further that the moduli spaces are generically reduced along these pullback components.
 Also, regarding logarithmic foliations $\omega=(\prod_{i=0}^{m} F_i) \sum_j \lambda_j \frac{dF_j}{F_j}$ we obtain another proof - although only in the case where $m\leq n+2$ and with some extra hypotheses on the degrees - of the fact proven in \cite{omegar}, \emph{i.e.}, that fixing the degrees $\deg (F_i)=d_i$ of the invariant divisors these logarithmic foliations define an irreducible component of $\FF^1(\PP^{n},\sum d_i)$.

\subsection{Statement of results}

This article deals with codimension $1$ foliations on projective spaces $\PP^n_\CC$ that are the pullback under rational maps of foliations defined on weighted projective spaces. We address the question of its stability under infinitesimal deformations and draw several consequences from it. To explain further the present work we introduce some definitions and notation.

It is \emph{a priori} unclear what a foliation is when the ambient space is singular. At the very least, we should ask a foliation on a singular space $X$ to restrict to a foliation on its maximal non-singular subscheme $X_r$. When $X$ is a normal and complete variety, which is the case we will consider here, $X_r\xrightarrow{j} X$ is an open subscheme whose complement $X\setminus X_r$ is a closed subset of codimension at least $2$. So whatever a codimension $1$ foliation on such an $X$ is, it should determine (up to multiplication by elements in $H^0(X_r,\OO_{X_r}^*)=\CC$) a global section  $\alpha\in H^0(X_r,\Omega^1_{X_r}\otimes \LL)$ for some line bundle $\LL$ on $X_r$ verifying integrability. Being $X$ normal, the restriction of divisors induces an isomorphism $Cl(X)\simeq Pic(X_r)$. This motivates the following definitions:

\begin{definition} Let $X$ be a normal projective variety and $\delta\in Cl(X)$. A \emph{codimension $1$ foliation on $X$ of degree $\delta$} is an element  $[\alpha]\in \PP  H^0(X_r,j_*(\Omega^1_{X_r}\otimes \OO_{X_r}(\delta)))$ satisfying
\begin{enumerate}
	\item $\sing(\alpha)=\{x \in X: \alpha(x)=0\}$ has codimension greater or equal than $2$ and
	\item $\alpha$ is \emph{integrable}, \emph{i.e.}, $\alpha\wedge d\alpha=0$.
\end{enumerate}
\end{definition}

\begin{remark} In order to maintain a clearer notation, we will write $\widehat{\Omega}^1_X(\delta)$ instead of $j_*(\Omega^1_{X_r}\otimes \OO_{X_r}(\delta))$. Observe that this is just a twisted version of the usual sheaf of Zariski differential $1$-forms $\widehat{\Omega}^1_X$. If $Y$ is another $\CC$-scheme, $\widehat{\Omega}^1_{X\times Y| Y}(\delta)$ will stand for the sheaf $(j_Y)_*(\Omega^1_{X_r\times Y| Y}\otimes \OO_{X_r}(\delta))$, where $j_Y:X_r\times Y\to X\times Y$ is the natural inclusion.
With respect to twisted vector fields, we will also use the notation $\T_X(\delta)=j_*(\T_{X_r}\otimes \OO_{X_r}(\delta))$.
\end{remark}

\begin{definition} Let $X$ be a complete normal variety and $\delta\in Cl(X)$.
The \emph{space of codimension $1$ singular foliations} of degree $\delta \in Cl(X)$ on $X$ is the quasi-projective variety
\[
\FF^1(X,\delta)=\{[\alpha]\in \PP H^0(X,\widehat{\Omega}^1_X(\delta)): \alpha \wedge d \alpha =0  \,\mbox{ and } \, \codim(\sing(\alpha))\geq 2   \}.
\]
\end{definition}

\begin{definition}
	Let $X$ be a complete normal variety, $\delta\in\mathrm{Cl}(X)$ and $[\alpha]\in \FF^1(X,\delta)$. Let $\Sigma=\spec(\CC[\varepsilon]/(\varepsilon)^2)$. A \emph{first order deformation} of $\alpha$ is an integrable element $\alpha_\varepsilon\in H^0(X\times \Sigma,\widehat{\Omega}^1_{X\times \Sigma|\Sigma}(\delta))$, modulo multiplication of invertible elements in $\CC[\varepsilon]/(\varepsilon)^2$, restricting to $\alpha$ on the central fiber.   Such an element must be of the form $\alpha_\varepsilon=\alpha + \varepsilon \beta$ for some  $\beta \in H^0(X,\widehat{\Omega}^1_X(\delta))/\langle \alpha \rangle$ such that
	\[
		(\alpha+\varepsilon \beta)\wedge d(\alpha+\varepsilon \beta)=0
	\]
	as sections of $\widehat{\Omega}^3_{X\times \Sigma|\Sigma}(2\delta)$, or equivalently \[\alpha\wedge d\beta + d\alpha\wedge\beta=0\] as sections of $\widehat{\Omega}^3_{X}(2\delta)$. Since the vector space of first order deformations of $\alpha$ identifies with the space $\T_{[\alpha]}\FF^1(X,\delta)$, it seems only fair to refer to the elements $\beta$ above as \emph{Zariski tangent vectors} of $\alpha$.
\end{definition}

\begin{definition}
	A foliation defined by a differential $1$-form $\alpha$ is said to have \emph{split tangent sheaf} if the sheaf of vector fields tangent to the foliation $\T_\alpha(U) = \{v\in \T_X(U) :\ i_v\alpha=0\}$ is isomorphic to a direct sum
	\[
		\T_\alpha \simeq \bigoplus_i \OO_X(\delta_i).
	\]
for some $\delta_i \in Cl(X)$.
\end{definition}
\medskip

Now, let $X=\PP^{m}(\overline{e})$ be a weighted projective space with weight vector $\overline{e}=(e_0,\dots,e_m)\in \NN^{m+1}$.

\begin{definition}
A foliation on $\PP^m(\ov{e})$ with split tangent sheaf is said to have a \emph{non-positive splitting} if each Weil divisor class $\delta_i$ corresponds to a non-positive integer under the usual isomorphism $Cl(\PP^m(\ov{e})) \simeq \ZZ$.
\end{definition}

By \cite[Theorem A]{V} if an integrable $1$-form $\alpha$ in $X$ has split tangent sheaf and the zeroes of $d\alpha$ form a variety of codimension greater than $2$, then any deformation of $\alpha$ has split tangent sheaf also. It is worth mentioning that although the referred result is stated for smooth toric varieties, its proof works just fine in the normal case.  This in turn implies that, in the moduli space of foliations on $X$, the generic member of an irreducible component containing $\alpha$ has split tangent sheaf (with the same splitting).

Recall that every rational map $F:\PP^n\dashrightarrow \PP^m(\ov{e})$ lifts in homogeneous coordinates to a map $F:\CC^{n+1}\setminus \{ 0\} \dashrightarrow \CC^{m+1}\setminus\{0\}$ defined by some unique (up to multiplication by scalars) homogeneous polynomials $F_0,\dots,F_m$ of degree $\deg(F_i)=ke_i$ for some $k\in \NN$ and vanishing simultaneously in codimension greater than $1$. In this case, we say that $F$ is of degree $k$. When no confusion arises, we will use the same notation for both the rational map and its polynomial lifting.
Now let $\delta\in \mathrm{Cl}(X)$ and $\alpha \in \PP H^0(X,\widehat{\Omega}^1_X(\delta))$ such that $\alpha\wedge d\alpha=0$.
For a generic rational map $F$ as above, the corresponding pullback $\omega=F^*(\alpha)$ defines a codimension $1$ foliation on $\PP^n$ of degree $k\delta$.

The main result of this article is the following.

\begin{theorem*}[Main Theorem]\label{elTeo}
	Let $X=\PP^{m}(\overline{e})$ be a weighted projective space, $\delta\in \mathrm{Cl}(X)$ and $\alpha \in \PP H^0(X,\widehat{\Omega}^1_X(\delta))$ defining a foliation with split tangent sheaf with non-positive splitting and such that $d\alpha$ vanishes in codimension greater than $2$.  Let $n\ge m+2$, $F:\PP^n\dashrightarrow X$ be a generic rational map of degree $k$ and $\omega:=F^*\alpha$. If $\eta\in H^0(\PP^n,\Omega^1_{\PP^n}(k\delta))$ defines a Zariski tangent vector of $\omega$ then there is a rational map $G:\PP^n\dashrightarrow X$ of degree $k$ and a form $\beta \in H^0(X,\widehat{\Omega}^1_X(\delta))$ defining a Zariski tangent vector of $\alpha$ such that
	\[
		\omega+\varepsilon \eta = (F+\varepsilon G)^* (\alpha + \varepsilon \beta).
	\]
	as sections of $\Omega^1_{\PP^n\times \Sigma|\Sigma}(k\delta)$, where  $\Sigma=\spec(\CC[\varepsilon]/(\varepsilon)^2)$.
\end{theorem*}

In other words, the above Theorem is stating that any first order deformation of $F^*\alpha$ is the pullback of a first order deformation of $\alpha$ by a deformation of $F$. As an immediate consequence, we can construct many irreducible components of the spaces of foliations on $\PP^n$:

\begin{corollary*}\label{elCor} Let $n\geq m+2$ and $\C\subseteq \FF^1(\PP^m(\overline{e}),\delta)$ be a generically reduced irreducible component whose generic element satisfies the hypotheses of \cref{elTeo}. Then for every $k\in \NN$ there exists an irreducible component $PB(n,k,\C)$ of $\FF^1(\PP^n,k\delta)$ whose generic element is a pullback of a foliation in $\C$ under a rational map $\PP^n\dashrightarrow \PP^m(\overline{e})$ of degree $k$. Moreover, the scheme $\FF^1(\PP^n,k\delta)$ is generically reduced along $PB(n,k,\C)$.
\end{corollary*}

In the case where the foliation on the base is induced by the infinitesimal action of a rigid Lie algebra $\g\subseteq H^0(\PP^m,\T_{\PP^m})$, we will combine the Corollary above with \cite[Theorem 3]{fj} as follows.

\begin{corollary*} Let $n\geq m+2$ and $\g\subseteq H^0(\PP^m,\T_{\PP^m})$ a rigid subalgebra of dimension $m-1$ such that $\exp(\g)$ acts with trivial stabilizers in codimension one. Let $\omega(\g)$ be the differential $1$-form defining the codimension $1$ foliation induced by this action. Suppose that $d\omega(\g)$ vanishes in codimension greater than $2$. Then the variety $PB(n,k,\g)\subseteq \FF^1(\PP^n,k(m+1))$ whose generic element is a pullback of $\omega(\g)$ under a rational map $\PP^n\dashrightarrow \PP^m$ of degree $k$ is an irreducible component of $\FF^1(\PP^n,k(m+1))$.
\end{corollary*}

\begin{table}[H]
\begin{center}
\begin{tabular}{| m{2.4cm} | m{0.8cm} | m{1.3cm} |m{4cm} | m{2cm} |}
\hline
Components & $n\ge$  & d & Generic element & Reference  \\
\hline
   $PB(n,k,\bar{e},\ell)$       &   $4$ & $k\ell$  & $\omega = F^*(\alpha)$,\hfill\ \linebreak  $F:\PP^n \dashrightarrow \PP^2(\bar{e})$ of degree k, $[\alpha] \in \PP H^0(\PP^2(\bar{e}),\widehat{\Omega}^1(\ell))$  and $ \ell$  as in \cref{condicionl}.
  & \cref{pullbackp2p} \\
  \hline
   $PBTM(n,s,\overline{e},\ell)$       &   $5$ & $ks$  & $\omega = F^*(\alpha)$,  $F:\PP^n \dashrightarrow \PP^3$ of degree k, $[\alpha]$ in the gen. reduced irred. comp. $TM_s(e_0,e_1,e_2,\ell)\subseteq \FF^1(\PP^3,s)$.
  & \cref{corPBTM} \\
    \hline
    $\mathcal{L}og(n,\bar{e})$     &$m+2$  & $\sum e_i$  & $\omega= (\prod F_i) \sum \lambda_j  dF_j / F_j$, $\deg(F_i)=e_i$, $\sum \lambda_i e_i=0$ and  $\bar{e}$ as in \cref{logcorolario}. & \cref{logexample} \\
      \hline
             $\mathcal{E}(n,k)$  &  5 &  $4k$ &  $\omega = F^*(\omega(\mathfrak{g}))$,\hfill\ \linebreak $F:\PP^n\dashrightarrow \PP^3$ of degree $k$ and $\mathfrak{g} = \mathfrak{aff}(\CC)$.  & \cref{expgeneral} \\
                  \hline
                        $PB(n,k,\mathfrak{g}(m))$ & $m+2$ & $k(m+1)$ & $\omega = F^*(\omega(\mathfrak{g}))$,\hfill\ \linebreak $F:\PP^n\dashrightarrow \PP^m$ of degree $k$ and $\mathfrak{g(m)} = \langle X, Y_1,\dots,Y_{m-1}\rangle$ with $[X,Y_i]= -2i Y_i$.  & \cref{pullbackgm}\\
         \hline
           $PB(n,k,\mathfrak{g}_6)$ & 8 & $7k$ & $\omega = F^*(\omega(\mathfrak{g}))$,\hfill\  \linebreak$F:\PP^n\dashrightarrow \PP^6$ of degree $k$, and $\mathfrak{g_6} = \langle X, Y_1,\dots,Y_{5}\rangle$ with $[X,Y_i]= -2i Y_i$, $[Y_1,Y_j] = Y_{j+1}$.   & \cref{pullbackg67} \\
         \hline
                  $PB(n,k,\mathfrak{g}_7)$ &  9 & $8k$ & $\omega = F^*(\omega(\mathfrak{g}))$,\hfill\ \linebreak $F:\PP^n\dashrightarrow \PP^4$ of degree $k$, and $\mathfrak{g_7} = \langle X, Y_1,\dots,Y_{6}\rangle$ with $[X,Y_i]= -2i Y_i$, $[Y_1,Y_j] = Y_{j+1}$, $[Y_2,Y_3] = -(5/2) Y_5$.  &\cref{pullbackg67}\\
         \hline
      \end{tabular}
      \caption{\small{Irreducible components of $\FF^1(\PP^n,d)$.}}
      \end{center}
\end{table}
We summarize in the table above all the families of irreducible components that we were able to derive from this work. Besides logarithmic foliations and the special cases mentioned in the introduction, all these stable families were previously unknown.

\section{Preliminaries}

Let us now briefly introduce some basic facts and conventions on weighted projective spaces. For complete proofs the reader is referred to \cite{DOLGACHEV} and \cite{CLS}. Along the rest of the article, $X$ will stand for the variety $\PP^m(\ov{e})$ with weight vector $\ov{e}=(e_0,\dots,e_m)$. This space can be presented as the quotient space of the action $\CC^*\curvearrowright \CC^{m+1}\setminus \{0\}$ defined by
\begin{equation}\label{accion}
  t\cdot (x_0,\dots,x_m)=(t^{e_0}x_0,\dots,t^{e_m}x_m).
\end{equation}
We will denote by $\pi_{\ov{e}}:\CC^{m+1}\setminus \{0\} \to \PP^m(\ov{e})$ the quotient morphism.

Let us fix some notation that we will use in the rest of the article: we will denote by $S$ the homogeneous coordinate ring of $\PP^n$,  $S=\CC[z_0,\ldots,z_n]$, and with a subscript $\ov{e}$ the corresponding graded polynomial ring of $\PP^m(\ov{e})$, that is $S_{\ov{e}}=\CC[x_0,\ldots,x_m]$. If no confusion arises, we will also denote with a subscript the homogeneous component of the given degree of $S_{\ov{e}}$ or $S$.

The tangent space to the fibers of $\pi_{\ov{e}}$ is pointwise spanned by the \emph{radial vector field}
\[ R_{\ov{e}}=\sum_{i=0}^m e_i x_i\frac{\partial}{\partial x_i}. \]
At the level of functions, the corresponding action of \cref{accion} diagonalizes into a $\ZZ$-grading of $S_{\ov{e}}=\CC[x_0,\dots,x_m]$ such that $\deg(x_i)=e_i$. This naturally induces a $\ZZ$-grading for both $\Omega^\bullet_{S_{\ov{e}}}$ and its dual $\bigwedge^\bullet Der(S_{\ov{e}},S_{\ov{e}})$ where $dx_i$ and $\frac{\partial}{\partial x_i}$ are of degree $e_i$ and $-e_i$ respectively.

It is a standard fact that the class group $Cl(X)$ is generated by the classes of the divisors $D_i=\pi_{\ov{e}}(\{x_i=0\})$. Moreover, the application $\deg:Cl(X)\to \ZZ$ such that $[\sum a_i D_i]\mapsto \sum a_ie_i$  defines an isomorphism $Cl(X)\simeq \ZZ$. Under this identification, the classes of Cartier divisors (\emph{i.e.}, the Picard group of $X$) identify with the free subgroup $\langle lcm(e_0,\dots,e_m)\rangle \subseteq \ZZ$. One of the good features of this grading is that the pullback under $\pi_{\ov{e}}$ induces an isomorphism
\[
H^0(X,\OO_X(\delta))\simeq (S_{\ov{e}})_\delta.
\]
We will often make use of both the Euler sequence
\begin{equation} \label{1}
0 \longrightarrow \widehat{\Omega}_{X}^{1} \longrightarrow \bigoplus_{i=0}^{m} \OO_X(-e_i) \longrightarrow \OO_X  \longrightarrow 0
\end{equation}
and its dual
\begin{equation}\label{2}
0 \longrightarrow \OO_X \longrightarrow \OO_X(e_i) \longrightarrow \T_X \longrightarrow 0
\end{equation}
(see for instance \cite{COXBATYREV}) in order to deal with the sheaves $\widehat{\Omega}^1_X(\delta)$ and $\T_X(\delta)$ as follows. Recall that the second arrow in the first sequence corresponds to the dot product with the element $(e_0 x_0,\dots,e_m x_m)$. For an element $\delta \in Cl(X)$, one can restrict \cref{1} to $X_r$ and tensorize with $\OO_{X_r}(\delta)$ and get a short exact sequence
\[
0 \longrightarrow \Omega_{X_r}^{1}\otimes\OO_{X_r}(\delta) \longrightarrow \bigoplus_{i=0}^{m} \OO_{X_r}(\delta-e_i) \longrightarrow \OO_{X_r}(\delta)  \longrightarrow 0.
\]
	Just as in the case of the projective space, taking global sections we see that  every $\omega \in H^0(X_r, \Omega^1_{X_r}\otimes \OO_{X_r}(\delta))$ (and therefore every global section of $\widehat{\Omega}^1_X(\delta)$) can be described in homogeneous coordinates as a polynomial differential form
 \begin{equation}\label{omega2}
\omega = \sum_{i=0}^m A_i(x)dx_i
\end{equation}
where $A_i\in (S_{\ov{e}})_{\delta-e_i}$ and $\imath_{R_{\ov{e}}}(\sum_{i=0}^m A_i(z)dz_i)=0$. In fact, the differential form above is just the pullback $\pi_{\ov{e}}^*\omega$. Since the same reasoning applies to higher differential forms, it is only natural to define the following.
\begin{definition} We will denote by
$$\widehat{\Omega}^\bullet_{S_{\ov{e}}}=\bigoplus_{\delta\in\ZZ} H^0(X,j_*(\Omega^\bullet_{X_r}\otimes O_{X_r}(\delta)))\subseteq \Omega^\bullet_{S_{\ov{e}}}$$
the \emph{graded module of differential forms descending to $X$}, \emph{i.e.}, the elements $\omega$ such that $\ii_{R_{\ov{e}}}(\omega)=0$.
\end{definition}
Regarding \cref{2}, with a similar argument we can see that every element $\X\in H^0(X,T_X(\delta))$ can be written as
\[ \X=\sum_{i=0}^m B_i(x)\frac{\partial}{\partial x_i} \]
where $B_i$ is homogeneous of degree $\delta+e_i$. This description is unique up to addition of multiples of the radial vector field.

\subsection{Idea of the proof}
Since the proof of the Main Theorem is quite technical, we will now give a quick overview of the ideas therein.

Let us start by spelling out what \cref{elTeo} is saying in terms of homogeneous coordinates of $\PP^n$. Let $F:\PP^n\dashrightarrow X$ be a generic rational map induced by some homogeneous polynomials $F_0,\dots,F_m$ of degree $\deg(F_i)=ke_i$, $\alpha=\sum_{i=0}^m A_i(x) dx_i\in H^0(X,\widehat{\Omega}^1_X(\delta))$ and $\omega=F^*\alpha$.
If $G=(G_0,\dots,G_m)$ is another homogeneous tuple of the same degrees, writing down the expression for $(F+\varepsilon G)^*\alpha$ in the ring $\CC[\varepsilon]/(\varepsilon)^2 [z_0,\dots,z_n]$ we get
    \begin{eqnarray*}
    	(F+\varepsilon G)^* \alpha = \sum_{i=0}^n A_i(F+\varepsilon G) d(F_i+\varepsilon G_i)=\\
    	\sum_{i=0}^n \left(A_i(F)+ \varepsilon \left(\sum_{j=0}^m \frac{\partial A_i}{\partial x_j}(F)G_j \right)\right) (dF_i+\varepsilon dG_i)=\\
    	\omega + \varepsilon\left( \sum_{i=0}^n A_i(F) dG_i + \sum_{i=0}^n \left(\sum_{j=0}^m \frac{\partial A_i}{\partial x_j}(F) G_j \right)dF_i\right).
    \end{eqnarray*}
 Since $(F+\varepsilon G)^*(\varepsilon\beta) = \varepsilon F^*\beta$, our main theorem is claiming that every Zariski tangent vector $\eta$ of $\omega$ can be decomposed as $\eta=\eta_1+\eta_2$, with $\eta_1=F^*\beta$ and
    \[
    	\eta_2 = \sum_{i=0}^n A_i(F) dG_i + \sum_{i=0}^n \left(\sum_{j=0}^m \frac{\partial A_i}{\partial x_j}(F) G_j \right)dF_i.
    \]
    Observe that both $\eta_1$ and $\eta_2$ here induce deformations of $\omega$, so we are actually decomposing the Zariski tangent space at $\omega$ into a direct sum of two subspaces.

In \cite[Theorem 4.14]{gmv} is proven that a Zariski tangent vector as $\eta_2$ above actually give rise to \emph{infinitesimal unfoldings} in the sense of \cite{moli}. As we will refer often to this particular type of deformations we make the following \emph{ad hoc} definition.

\begin{definition}\label{defSUnf}
	Let $\omega= F^*\alpha = \sum_i A_i(F) dF_i$ be a pullback form as in the main theorem.
	We say a a deformation of $\omega$ is a \emph{special unfolding} if it is given by a Zariski tangent vector of the form
	\[
    	\eta_2 = \sum_{i=0}^n A_i(F) dG_i + \sum_{i=0}^n \left(\sum_{j=0}^m \frac{\partial A_i}{\partial x_j}(F) G_j \right)dF_i.
    \]
    If no confusion arises, we will refer as \emph{special unfolding} to both the deformation of $\omega$ and the Zariski tangent vector inducing such deformation.
\end{definition}
 In \cite[Theorem 4.14]{gmv} is proven that, under suitable hypotheses, every unfolding of $\omega$ comes from a form like $\eta_2$ above, so sometimes we can conclude that every unfolding is special in the above sense. However we will not need results of this type, as we will prove that every deformation of $F^*\alpha$ is decomposed as a sum of the pullback by $F$ of a deformation of $\alpha$ and a deformation given by a special unfolding.

The first step in establishing the main theorem is the realization of the following fact.

\begin{proposition}\label{propPullback}
	Let $X$, $\omega$, $\alpha$ and $F:\PP^n\dashrightarrow X$ be as in \cref{elTeo}. Let $\eta$ be a Zariski tangent vector of $\omega$.
	If $\eta$ can be written as
	\[
		\eta = \sum_{j=0}^m B_j(z) dF_j
	\]
	then there is a Zariski tangent vector $\beta$ of $\alpha$ such that $\eta=F^*\beta$.
\end{proposition}

In view of the above we need to prove that a Zariski tangent vector $\eta$ is equivalent, as an element of $\Omega^1_S/\langle dF_0,\dots, dF_m\rangle$, to a special unfolding.
To do this we study the restriction of  $\eta$ to a certain subscheme of singularities of $\omega$.

\begin{remark}	Notice that, in general, $dF_i$ is well-defined only as a form in affine space $\mathbb{A}^{n+1}$ but does not descend to a form in $\PP^n$. Therefore, in view of \cref{propPullback}, it will be convenient to work with the free sheaf $\OO_{\PP^n}^{\oplus n+1}$ associated to $\Omega^1_S$ and generated by sections $dz_0,\dots, dz_n$.
\end{remark}

\begin{definition}\label{defK0}
	Let $\omega=\sum_{i=0}^m A_i(F)dF_i$ be as in \cref{elTeo}. We define \emph{$K_0$} to be the ideal generated by $\langle A_0(F),\dots, A_m(F)\rangle$ of $S$, and \emph{$\mathscr{V}_0\subset \PP^n$} to be the subscheme defined by $K_0$.
\end{definition}

Since under the hypotheses of \cref{elTeo}, a generic point $\pp$ of $\VV_0$ is a \emph{Kupka point}, meaning that $\omega|_\pp=0$ and $d\omega|_\pp\neq 0$, when we restrict the deformation condition of $\eta$ to $\VV_0$ we get
\[
	\left(\omega\wedge d\eta + d\omega\wedge\eta\right)|_{\VV_0}= 0
\]
and so
\[
 	\eta|_{\VV_0}\wedge d\omega|_{\VV_0}=  0.
\]

This immediately implies that $\eta|_{\VV_0} \wedge dF_0|_{\VV_0}\wedge\dots\wedge dF_m|_{\VV_0} = 0$ which, together with \cref{Lema1}, allow us to apply Saito's division theorem \cite{saito} to conclude that $\eta|_{\VV_0}$ is a section of the subsheaf of $\OO_{\VV_0}^{\oplus n+1}$ generated by $\langle dF_0|_{\VV_0},\dots, dF_m|_{\VV_0}\rangle$.
Stated in terms of the $S$-module $\Omega^1_S$ this simply implies that $\eta$ can be written as
\begin{equation}\label{ecu-10}
	\eta= \sum_{i=0}^m A_i(F)\gamma_i + \sum_{j=0}^m B_j dF_j,
\end{equation}
for some $\gamma_i \in \Omega^1_S$.

Now the last step of the proof  is to show that the elements $\gamma_i$ above can be taken to be of the form $dG_i$ for some homogeneous polynomials $G_i$ at the cost of adding some special unfoldings on the right side of \cref{ecu-10}.
This, combined with \cref{propPullback}, essentially concludes the proof of the main theorem as is immediate from \cref{defSUnf} that any element of $\Omega^1_S$ of the form $ \sum_{i=0}^m A_i(F) dG_i$ is equivalent modulo $\langle dF_0,\dots,dF_m\rangle$ to a special unfolding of $\omega$.

\section{Proof of main Theorem.}

Let $X=\PP^m(\overline{e})$ be a weighted projective space with $m\geq 2$, $F:\PP^n\dashrightarrow X$ a dominant rational map and $\alpha$ an homogeneous integrable differential $1$-form on $X$ of total degree $\delta \in Cl(X)$. Recall that $R_{\ov{e}}:=\sum_{i=0}^m e_ix_i \frac{\partial}{\partial x_i}$ is the radial vector field of $X$ and set $\Lambda_X:= i_{R_X} dx_0\wedge\cdots\wedge dx_m$. In order to carry out the proof of our main result, let us start with some lemmas and propositions, some of them of independent interest. The following result is essential in the proof of \cref{propPullback}.

\begin{proposition}
	 Let $\eta$ be a Zariski tangent vector of $\omega = F^*\alpha$ as an integrable form. The following are equivalent.
	\begin{enumerate}
		\item $\eta\wedge F^*(\Lambda_X)=0$.

		\item There is a Zariski tangent vector $\beta$ of $\alpha$ as an integrable form such that $\eta=F^*\beta$.
	\end{enumerate}
\end{proposition}
\begin{proof}
The latter condition easily implies the first one. Let us focus on the other implication. Let $\eta$ be an element satisfying the first statement. Straightforward calculation shows that the kernel of the first order deformation $\omega_{\varepsilon} = \omega + \varepsilon \eta$ is tangent to the fibers of the trivial deformation $F_{\varepsilon} = \PP^n \times \Sigma \dashrightarrow X \times \Sigma$. Now the Proposition follows from \cite[Theorem 4.6]{gmv}.
\end{proof}

\begin{proof}[Proof of \cref{propPullback}]
By the above result, we only need to show that $\eta\wedge F^*(\Lambda_X)=0$. In this case we have
	\[
		\eta\wedge F^*(\Lambda_X)= \left(\sum_k B_k e_k F_k \right) dF_0\wedge\cdots\wedge dF_m = i_R(\eta) dF_0\wedge\cdots\wedge dF_m =0.
	\]
	Where $R$ is the radial field of $\PP^n$ and $i_R \eta=0$ as $\eta$ is a Zariski tangent vector of $\omega$.
\end{proof}

Now we will define some ideals of the polynomial ring $S=\CC[z_0,\dots, z_n]$ that will be useful in the proof of the main theorem.

\begin{definition}
        Let $\omega = F^*\alpha=\sum_i A_i(F)dF_i$ be an integrable differential $1$-form on $\PP^n$ as in \cref{elTeo}. Also let $M$ be the jacobian matrix of $F$, that is a matrix of size $(n+1)\times (m+1)$ with polynomial coefficients.
        We define the following ideals of $S$:
        \begin{itemize}
			\item $\m := \langle z_0,\ldots,z_n\rangle$.
        	\item $B(F) := \langle F_0,\dots, F_m\rangle $.
        	\item $K_r := K_0 \cdot B(F)^r$.
        	\item $I$ to be the ideal generated by the maximal minors of $M$.
        \end{itemize}
\end{definition}

The following lemma will allow us to carry out division arguments for differential forms on the schemes defined by the ideals $K_r$. For the definition of \emph{depth} and basic properties related to it we refer to \cite{MATSU}.

\begin{lemma}\label{Lema0}
For a generic map $F:\PP^n\dashrightarrow X$ and a generic integrable differential $1$-form $\alpha$ on $X$ we have
\[ \depth_S(I,S/K_r)\geq n-m .\]
\end{lemma}
\begin{proof} Let us first observe that if a pair $(F,\alpha)$ satisfies
\begin{enumerate}[label={\bfseries \arabic*.}]
	\item the critical points of $F$ do not intersect $V(K_0)$,
	\item $\depth_{S_\m}(\langle z_0,\ldots,z_n\rangle_\m,(S/K_r)_\m)\geq n-m$,
\end{enumerate}
then $\depth_S(I,S/K_r)\geq n-m $. Recall that
\[ \depth_S(I,S/K_r)=\depth_{S/K_r}(\overline{I}, S/K_r)\ , \]
where $\overline{I}$ denotes the class of $I$ in the quotient $S/K_r$, and also
\[ \depth_{S/K_r}(\overline{I},S/K_r)=\inf \{\depth_{(S/K_r)_\pp}(S/K_r)_\pp; \pp \in V(\overline{I}) \}. \]
Condition \textbf{1.} implies that the closed subscheme defined by $I$ on $V(K_r)\subseteq \CC^{n+1}$ is supported at the origin. But then
\begin{align*}
\depth_S(I,S/K_r)&=\depth_{S_\m}(\overline{I_\m},(S/K_r)_\m) \\
&= \depth_{S_\m}(\overline{\sqrt{I_\m}}, (S/K_r)_\m) \\
&= \depth_{S_\m}(\langle z_0,\dots,z_n\rangle_\m, (S/K_r)_\m)
\end{align*}
which is greater or equal than $n-m$ by condition \textbf{2}.

Now both properties \textbf{1.} and \textbf{2.} define dense open sets in the space of pairs $(F,\alpha)$: for the first condition one can take any $\alpha$ defining a foliation on $X$ and a simple normal crossing $F=(F_0,\dots,F_m)$ whose critical values do not intersect the singularities of $\alpha$. The second condition corresponds to the vanishing of $Ext^k_S(S/\langle z_0,\dots,z_n\rangle,S/K_r)$ for $k> n-m$. Since the dimension of this space is upper semicontinuous as a function of $(F,\alpha)$ one can conclude that this condition is also Zariski dense.
\end{proof}

It will be convenient to use the following notation. For $F_0,\dots, F_m$ as above and $J\in \{0,\dots,m\}^r$ we write
\[F_J = \prod_{l=1}^r F_{J_l}.\]

\begin{lemma}\label{Lema1}
Let $\alpha$ be a generic integrable differential 1-form on $X$, $r\geq 0$, $\eta$ be an element of the module $\Omega^1_S$ and $V(K_r)\subseteq \PP^n$ the subscheme associated with the ideal $K_r$. Let $n\ge m+2$. If for some generic homogeneous polynomials $F_0,\dots, F_m$ we have $\eta|_{V(K_r)}\wedge dF_0|_{V(K_r)}\wedge\cdots\wedge dF_m|_{V(K_r)}=0$ then we can express $\eta$ as
	\[
		\eta = \sum_{i=0}^m\sum_{\substack{ J\in \{1,\dots,m\}^r}} A_i(F) F_J \eta_i^J + \sum_{j=0}^m C_j dF_j,
	\]
	for some $\eta_i^J \in \Omega^1_S$.
\end{lemma}
\begin{proof}
Recall that $\depth_S(I,S/K_r)=\depth_{S/K_r}(\overline{I}, S/K_r)$. From \cref{Lema0} we know that the ideal $\overline{I}$ of coefficients of $dF_0|_{V(K_r)}\wedge\cdots\wedge dF_m|_{V(K_r)}$ in the ring $S/K_r$ has depth greater than one.  Finally, the result immediately follows from Saito’s generalization of De Rham’s division Lemma, see \cite{saito}.
\end{proof}

\begin{lemma}\label{Lema2}
	Let $\nu \in \Omega^2_S$ be a homogeneous element of degree $l$ and $F_0,\dots, F_m$ be generic homogeneous polynomials of degree $ke_0,\dots, ke_m$ respectively. If $\nu\wedge dF_0\wedge\cdots\wedge dF_m=0$ then $\nu=\sum_{i=0}^m \nu_i\wedge dF_i$, for some $\nu_i\in\Omega^1_S$ of degree $l-ke_i$.
\end{lemma}
\begin{proof}
Assuming the polynomials $F_i$ satisfy the generic condition given by $\depth_S(I,S)\geq 3$, the result is again a consequence of \cite{saito}.
\end{proof}

The hypothesis on $\alpha$ being a split foliation in \cref{elTeo} is in order to have a good resolution of the ideal defining its singularities. The following proposition shows that if $F$ is generic then this induces a good resolution of the ideal $K_0$.

\begin{proposition}\label{propACM}
	With the notation above, consider
	\[\alpha=i_{X_1}\cdots i_{X_{m-1}}i_{R_X} \ dx_0\wedge\cdots\wedge dx_m = \sum_{i=0}^m A_i(x) dx_i\in H^0(X,\widehat{\Omega}^1_X(\delta))
	\]
	defining a foliation with split tangent sheaf generated by the vector fields $X_j= \sum_i B^j_i(x) \frac{\partial}{\partial x_i}.$
	Suppose $F:\PP^n\dashrightarrow X$ is a rational map of degree $k$ such that the singularities of $\alpha$ are regular values for $F$, \emph{i.e.} if $p\in\sing(\alpha)\subset  X$ then $F^{-1}(p)$ is a regular variety outside the base locus of $F$. Let $\omega=F^*\alpha$.
	Then $\VV_0= V(K_0)$ is an arithmetically Cohen-Macaulay subscheme of $\PP^n$ and the ideal $K_0$ has a two step resolution
	\[
	\xymatrix{
	0 \ar[r]^{} & \bigoplus_{r=0}^{m-1}S(d_r) \ar[r]^-{M}  & \bigoplus_{i=0}^m S(k(e_i-\delta))   \ar[rr]^-{(A_i(F))_{i=0}^m}  & & K_0 \ar[r] & 0\\
	}
	\]
where $d_r\in\ZZ$ conveniently taken and $M$ is the matrix defined as
\[
M=
		\left(\begin{array}{ccc}
		      	e_0F_0 & \cdots & e_mF_m\\
		      	B^1_0(F) & \cdots & B^1_m(F) \\
		      	 & \vdots & \\
		      	 B^{m-1}_0(F) & \cdots & B^{m-1}_m(F)
		      \end{array}\right)
\]
\end{proposition}

\begin{proof}
The hypothesis on $F$ implies that $\VV_0$ has codimension $\ge2$. Also notice that $K_0$ is exactly generated by the minors of $M$. Then the result follows from the Hilbert-Schaps Theorem, see \cite[Theorem 5.1]{Artin76}.
\end{proof}

\begin{lemma}\label{Lema3}
	Let $F=(F_0,\dots, F_m):\PP^n\dashrightarrow \PP^m(\overline{e})$ be generic rational map of degree $k$ and  $\sum_{J\in \{1,\dots,m\}^r}g_J F_J =0$ be a relation between the generators $F_J$ of the ideal $B(F)^r$. Then each $g_J$ is in $B(F)$.
\end{lemma}

\begin{proof}
Actually, the result can be stated in terms of a regular sequence of $m+1$ non-zero homogeneous elements $F_0, \dots ,F_m$ in a graded integral domain $S$ and  proved by induction on $m$. Assume it is true for sequences of $m$ elements and consider a regular sequence $F_0,\dots,F_m$ in $S$. Working on $S/\langle F_0 \rangle$, we have
\[
\sum_{\substack{J\in \{1,\dots,m\}^r \\ 0 \notin J}}\overline{g_J} \overline{F_J} =0.
\]
Since $\overline{F_1}, \dots, \overline{F_m}$  is again a regular sequence, then each $\overline{g_j} \in \langle \overline{F_1},\dots, \overline{F_m}\rangle$ and the result follows easily.
\end{proof}

The following technical result will be a central tool to perform an iterative argument in the proof of \cref{elTeo}.

\begin{proposition}\label{propPasoRecur}
	Let $r\geq 0$ and $\omega=F^*\alpha=\sum_{i=0}^m A_i(F) dF_i$ be as in \cref{elTeo}. Suppose
	\[
		\eta=\sum_{i=0}^m \sum_{J\in \{0,\dots,m\}^r} A_i(F) F_J \eta^i_J + \sum_{k=0}^m C_k dF_k ,
	\]
	 is a Zariski tangent vector of $\omega$.
	Then we can write $\eta$ as
	\[
		\eta= \widetilde{\eta} + \sum_{i=0}^m \sum_{J'\in \{0,\dots,m\}^{r+1}} A_i(F) F_{J'} \eta^i_{J'}+\sum_{k=0}^m D_k dF_k,
	\]
with $\widetilde{\eta}$ a special unfolding and $\eta^i_{J'} \in \Omega^1_S$ such that $\deg \eta^i_{J'}=e_i-\sum_{j\in J'} e_j$.
\end{proposition}

\begin{proof}
	As $\eta$ is a Zariski tangent vector of $\omega$ then the following equation is verified.
	\begin{equation}\label{eqdef}
		\eta\wedge d\omega + \omega\wedge d\eta=0.
	\end{equation}
	Multiplying \cref{eqdef} by $dF_0\wedge\cdots\wedge dF_{k-1}\wedge dF_{k+1}\wedge\cdots\wedge dF_m$ we get
	\[
		\left(\sum_{i=0}^m\sum_{J\in \{0,\dots,m\}^r} A_i(F)F_J d\eta^i_J\right) A_k(F) dF_1\wedge\cdots\wedge dF_m =0.
	\]
	Then, by \cref{Lema2} there are homogeneous elements $\nu_0,\dots,\nu_m\in \Omega^1_S$ such that
	\begin{equation}\label{eqII}
		\sum_{i=0}^m\sum_{J\in \{0,\dots,m\}^r} A_i(F)F_J d\eta^i_J = \sum_{k=0}^m \nu_k\wedge dF_k.
    \end{equation}
    Now in the left hand side of \cref{eqII} all the coefficients multiplying the forms $d\eta^i_J$ are in the ideal $K_r$.
    So if we restrict \cref{eqII} to $V(K_r)$ and we multiply by $\bigwedge_{i\neq k} dF_i$ we get
    \[
    	\nu_k|_{V(K_r)}\wedge dF_0|_{V(K_r)}\wedge\cdots\wedge dF_m|_{V(K_r)}=0.
    \]
    Then by \cref{Lema1} we have the following equation in the module $\Omega^1_S/\langle dF_i\rangle_{i=0}^m$
    \[
    	\nu_k\equiv \sum_{i=0}^m\sum_{J\in \{0,\dots,m\}^r} A_i(F)F_J \nu_{i,J}^k \mod \langle dF_0,\dots,dF_m\rangle
    \]
    Replacing in \cref{eqII} the $\nu_k$ by this expression we get
    \[
    	\sum_{i=0}^m\sum_{J\in \{0,\dots,m\}^r} A_i(F)F_J d\eta^i_J =
    	\sum_{i,k=0}^m\sum_{J\in \{0,\dots,m\}^r} A_i(F)F_J\nu^k_{i,J}\wedge dF_k+\sum_{\substack{k,l=0\\k<l}}^mB_{kl} dF_k\wedge dF_l,
    \]
    for some homogeneous polynomials $B_{kl}$.
    If we restrict once more to $V(K_r)$ and multiply by $\bigwedge_{i\neq k,l}dF_i$, we can deduce $B_{lk}|_{V(K_r)}=0$, \emph{i.e.} $B_{lk}\in K_r$.
    As $K_r$ is generated by the polynomials $A(F)F_J$ with $|J|=r$ then we can regroup terms to get forms $\widetilde{\nu^k_{i,J}}$ such that
    \begin{equation}\label{eqIII}
    	\sum_{i=0}^m\sum_{J\in \{0,\dots,m\}^r} A_i(F)F_J d\eta^i_J = \sum_{i,k=0}^m\sum_{J\in \{0,\dots,m\}^r} A_i(F)F_J\widetilde{\nu^k_{i,J}}\wedge dF_k.
    \end{equation}
    Contracting \cref{eqIII} with the radial vector field $R=\sum_i \frac{\partial}{\partial z_i}$ we get
    \[
    	\sum_{i=0}^m A_i(F) \left(\sum_{J\in \{0,\dots,m\}^r}F_J i_R\left(d\eta^i_J-\sum_{k=0}^m \widetilde{\nu_{i,J}^k}\wedge dF_k \right)\right)=0.
    \]
    Lets call $\mu_i:= \sum_{J\in \{0,\dots,m\}^r}F_J i_R\left(d\eta^i_J-\sum_{k=0}^m \widetilde{\nu_{i,J}^k}\wedge dF_k \right)$, then $\mu_i$ is an element of $\Omega^1_S$ of homogeneous degree $e_i$, which means that the coefficients of $\mu_i$ are polynomials of degree $e_{i}-1$.
    But at the same time the coefficients of the $\mu_i$'s define relations among the $A_i(F)$.
    By \cref{propACM} these relations must be all trivial, because $\alpha$ has a non-positive splitting and then  no non trivial relations among the $A_i(F)$'s can have a term multiplying $A_i(F)$ of degree less than $e_i$.
    Therefore we have $\mu_i=0$ for all $i=0,\dots, m$.
    We can further decompose $\mu_i$ as a sum by defining $\mu_{i,J}:= i_R\left(d\eta^i_J-\sum_{k=0}^m \widetilde{\nu_{i,J}^k}\wedge dF_k \right)$ and get
    \[
    	\mu_i = \sum_{J\in \{0,\dots,m\}^r} F_J \mu_{i,J} =0.
    \]
    Where now $\deg \mu_{i,J} = e_i-\sum_{j\in J} e_j$.
    Now as the coefficients of the $\mu_{i,j}$ are relations among the $F_J$ with $|J|=r$ we have by \cref{Lema3} that $\mu_{i,J}= \sum_{l=0}^m F_l \mu_{i,J}^l$.
    We have then
    \[
    	i_R\left(d\eta^i_J-\sum_{k=0}^m \widetilde{\nu_{i,J}^k}\wedge dF_k \right)= \sum_{l=0}^m F_l \mu_{i,J}^l.
    \]
    Which implies
    \[
    	\eta^i_J = d\left( -\frac{1}{a_J} i_R\eta^i_J\right) + \sum_{k=0}^m \frac{1}{a_J} i_R \left(\widetilde{\nu^k_{i,J}}\right) dF_k + \sum_{l=0}^m F_l \widetilde{\mu_{i,J}^l},
    \]
    with $a_J$ an integer number and some elements  $\widetilde{\mu_{i,J}^l}\in \Omega^1_S$.

    Now, we can define $G^i_J :=-\frac{1}{a_J} i_R\eta^i_J$ and $H_i:=\sum_{J\in \{0,\dots,m\}^r}F_JG^i_J$. Then taking our original Zariski tangent vector \[\eta=\sum_{i=0}^m \sum_{J\in \{0,\dots,m\}^r} A_i(F) F_J \eta^i_J + \sum_{k=0}^m C_k dF_k\] and writing down everything we have done so far,  we finally get
    \[
    	\eta = \widetilde{\eta} + \sum_{i=0}^m\sum_{J\in \{0,\dots,m\}^r} \sum_{l=0}^m A_i(F) F_J F_l \widetilde{\mu_{i,J}^l} + \sum_{k=0}^m \tilde{C}_k dF_k,
    \]
    with
    \[
        \widetilde{\eta} = \sum_{i=0}^m A_i(F) dH_i + \sum_{i=0}^m \sum_{j=0}^m \left(\frac{\partial A_i}{\partial x_j}(F) H_j\right) dF_i.
    \]
    which gives us the desired result.
\end{proof}

Now we are ready to complete the proof of our main result.

\begin{proof}[Proof of \cref{elTeo}]
	If $\eta$ is a Zariski tangent vector of $\omega$ then
	\[
		\eta|_{\VV_0}\wedge d\omega|_{\VV_0}=0.
	\]
	As a generic point of $\VV_0$ is a Kupka point, then $d\omega|_{\VV_0}\neq 0$ which implies
	\[
		\eta|_{\VV_0}\wedge dF_0\wedge\cdots\wedge dF_m|_{\VV_0}=0.
	\]
	Now by \cref{Lema1} this implies that we can write
	\[
		\eta= \sum_{i=0}^m A_i(F) \eta_i + \sum_{k=0}^m C_k dF_k.
	\]
	Now we can iteratively apply  \cref{propPasoRecur} to write
	\[
		\eta = \sum_{j=1}^N \eta_j + \text{something in } \langle dF_0,\dots,dF_m\rangle.
	\]
	where $\eta_j$ is a special unfolding for $j=1,\dots,N$.
	Notice that $\eta-\sum_j \eta_j$ is still a Zariski tangent vector of $\omega$ that belongs to the submodule generated by $dF_0,\dots,dF_m$. Finally,  using \cref{propPullback}, we can write $\eta$ as a sum of special unfoldings and the pullback of a Zariski tangent vector of $\alpha$, as claimed.
\end{proof}

\section{Applications}\label{sectionApp}

This section is devoted to construct new irreducible components of the spaces of foliations $\FF^1(\PP^n,d)$ by means of applying  \cref{elCor}. We will also show that the stability of some families already appearing in the literature can be deduced from this result as well.

\subsection{Pullbacks of foliations on $\PP^2(\overline{e})$}
In the case $m=2$, the integrability condition is trivial and therefore the spaces $\FF^1(\PP^2(\overline{e}),\delta)$ are open in $\PP H^0(\PP^2(\overline{e}),\widehat{\Omega}^1_{\PP^2(\overline{e})}(\delta))$. In particular, they are irreducible and reduced. Since every foliation on a surface is defined by a single twisted vector field, the tangent sheaf of every foliation will have a non-positive splitting whenever $\delta\geq e_0+e_1+e_2$. This enables us to apply \cref{elCor} to the case $\C=\FF^1(\PP^2(\overline{e}),\delta)$ provided that a generic element in $\C$ admits only Kupka singularities:

\begin{corollary}\label{pullbackp2p} Let $\ov{e}=(e_0,e_1,e_2)$, $n\geq 4$ and $\delta\geq e_0+e_1+e_2$ such that a generic foliation of degree $\delta$ on $\PP^2(\overline{e})$ has only Kupka singularities. Then for every $k\in \NN$ there exists an irreducible component $PB(n,k,\overline{e},\delta))$ of $\FF^1(\PP^n,k\delta)$ whose generic element is a pullback  of a foliation of degree $\delta$ on $\PP^2(\overline{e})$   under a rational map $F:\PP^n \dashrightarrow \PP^2(\overline{e})$ of degree $k$. Moreover, the scheme $\FF^1(\PP^n,k\delta)$ is generically reduced along $PB(n,k,\overline{e},\delta))$.
\end{corollary}
\begin{remark}Since the hypothesis on the singularities is clearly satisfied in the case of $\PP^2(\ov{e})=\PP^2$, this gives an alternative proof of the main Theorem in \cite{CLE} for $n\geq 4$.
\end{remark}

It could be the case, however, that for an arbitrary $\ov{e}$ every foliation of some special degree $\delta$ contains non-Kupka points among their singularities. Let us now construct for every $\overline{e}$ an infinite family of degrees where this hypothesis is satisfied.

Depending on the arithmetic of $\ov{e}$ and $\delta$, some of the points $[1:0:0]$, $[0:1:0]$ and $[0:0:1]$ may be forced to be a singularity of \textit{every} foliation of degree $\delta$. Moreover, such points will be non-Kupka singularities. In \cite[Example 3.9]{gmv} it is shown that the degrees $\delta$ avoiding this phenomenon are the ones such that for every $i$ we have
\begin{equation}\label{equcong} e_i | \delta-(e_0+e_1+e_2) + e_j
\end{equation}
for some $j$. The solutions to the equations consist of a certain (positive) number of residues modulo $e_0e_1e_2$. Observe further that for such $\delta$ the canonical map
$$H^0(\PP^2(\ov{e}), \T_{\PP^2(\ov{e})}(\delta-(e_0+e_1+e_2)))\otimes_\CC\OO_{\PP^2(\ov{e})}\to \T_{\PP^2(\ov{e})}(\delta-(e_0+e_1+e_2))$$
is non-zero at every fiber. In particular, applying \cite[Proposition 3.7]{gmv} we can deduce:

\begin{lemma}\label{condicionl} Let $\delta\geq e_0+e_1+e_2$ satisfying conditions (\ref{equcong}) above. Then for every $k\geq 1$ a generic foliation of degree $\delta+k e_0e_1e_2$ will have only Kupka singularities.
\end{lemma}

Until this point we have only dedicated ourselves to construct stable families which are pullbacks under \emph{generic} rational maps. However, in some situations our methods also enable us to deal with the non-generic case as well.

In \cite[Section 4]{degree3LPdC} the authors construct subvarieties $TM_s(e_0,e_1,e_2,l)$ of the space $\FF^1(\PP^3,s)$ whose general member is conjugated to a foliation tangent to a multiplicative action of the form
\begin{align*}
    \phi^{(e_0,e_1,e_2)}:\CC^*\times \PP^3 &\to \PP^3 \\
    (\lambda,[x_0,\dots,x_3])&\to [\lambda^{e_0} x_0, \lambda^{e_1} x_1, \lambda^{e_2} x_2,x_3],
\end{align*} and satisfying
\[ L_{v_{(e_0,e_1,e_2)}}\omega_\FF= l\, \omega_\FF,\]
where $v_{(e_0,e_1,e_2)}= e_0 x_0\frac{\partial}{\partial x_0} + e_1 x_1 \frac{\partial}{\partial x_1} + e_2 x_2 \frac{\partial}{\partial x_2}$ is a vector field generating the $\CC^*$-action.
The general leaf of these foliations happens to be a pullback of a foliation on $\PP^2(\ov{e})$ under the rational map
\begin{align*}
   \pi_{\ov{e}}: \PP^3&\dashrightarrow\PP^2_{(e_0,e_1,e_2)}\\
    [x_0:\dots:x_3]&\mapsto [x_0 x_3^{e_0-1}: x_1 x_3^{e_1-1}: x_2 x_3^{e_2-1}].
\end{align*}
Moreover,  \cite[Corollary 4.11]{degree3LPdC} (and its corresponding proof) states that if $1\leq e_0 < e_1 < e_2$ are coprime integers such that the general member of $TM_s(e_0,e_1,e_2,l)$ has finitely many non-Kupka singularities, then its tangent sheaf is of the form
$$\T\FF\simeq \OO_{\PP^3}\oplus \OO_{\PP^3}(2-s)$$
and $TM_s(e_0,e_1,e_2,l)$ is an irreducible component of $\FF^1(\PP^3,s)$. Observe that the splitting above is non-positive. The numerous families of tuples $(e_0,e_1,e_2,l)$ such that the general member of $TM_s(e_0,e_1,e_2,l)$
has only finitely many non-Kupka singularities are parametrized in \cite[Theorem 4.12]{degree3LPdC}. As it is shown in the mentioned work, part of this family of components will be generically reduced, although it is not clear how to isolate this subset with a method other than a case by case study.
We are now ready to use \cref{elCor}.

\begin{corollary}
\label{corPBTM}    Let $TM_s(e_0,e_1,e_2,l)$ be a generically reduced irreducible component of $\FF^1(\PP^3,s)$ whose general member has finitely many non-Kupka singularities. Then for every $k\in \NN$ there exists an irreducible component $PBTM(n,s,\ov{e},l)\subseteq \FF^1(\PP^n,ks)$ whose general element is a pullback of a foliation in $TM_d(e_0,e_1,e_2,l)$ under a map $\varphi:\PP^n\dashrightarrow\PP^3$ of degree $k$. In particular, these foliations are pullbacks under the non-generic rational maps $\pi_{\overline{e}}\circ\varphi: \PP^n\dashrightarrow\PP^2(\ov{e})$.
\end{corollary}

\subsection{Pullbacks of foliations induced by Lie group actions}

Let us now move to a different type of objects, namely foliations on $\PP^m$ whose tangent sheaf is trivial. The corresponding integrable differential $1$-form can be constructed as follows. Let $\g\subseteq H^0(\PP^m,\T_{ \PP^m})$ be a $(m-1)$-dimensional Lie subalgebra such that $\exp(\g)\subseteq Aut(\PP^m)$ acts with trivial stabilizers outside some algebraic set of codimension $2$. Since $\bigwedge^{m-1}\g \subseteq \bigwedge H^0(\PP^m, \T_{\PP^m})$ is a one-dimensional subspace, we obtain by duality an integrable element
\[ \omega(\g)\in H^0(\PP^m,\Omega^1_{\PP^m}(m+1)) .\]
The leaves of the foliation induced by $\omega(\g)$ coincide with the generic orbits of the action of $\exp(\g)$ and the corresponding tangent sheaf satisfies
$$\T_{\omega(\g)}\simeq \g \otimes_\CC\OO_{\PP^m}.$$
Along the literature one can find many stable families of foliations of this kind. We will now apply \cref{elCor} to some of these families in order to produce irreducible components of the spaces of foliations. First, let us recall `Theorem 3' in \cite{fj}, which can be stated as follows:

\begin{theorem} \label{teoCP}
Let $\g\subseteq H^0(\PP^m,\T_{\PP^m})$ be a Lie subalgebra such that $\exp(\g)$ acts with trivial stabilizers in codimension one and such that $d\omega(\g)$ vanishes in codimension greater than $2$. Then for every foliation $\FF'$ near $\FF(\g)$ there exists a subalgebra $\g'\subseteq H^0(\PP^m, \T_{\PP^m})$ near $\g$ such that $\FF'=\FF(\g)$. In particular, if $\g$ is a rigid subalgebra, then $\overline{Aut(\PP^m)\cdot \FF(\g)} $ is an irreducible component of the space $\FF^1(\PP^m,m+1)$.
\end{theorem}

\begin{example}\label{logexample} We shall now give an alternative proof of the stability of \textit{logarithmic} foliations on $\PP^n$ with $m\leq n-2$ divisors. This is, foliations which are induced by a differential $1$-form
\begin{equation}\label{log1} \omega=\sum_{i=0}^m \lambda_i \widehat{F_i} dF_i \in H^0(\PP^n,\Omega^1_{\PP^n}(d)),
\end{equation}
where  $F_i\in H^0(\PP^n,\OO_{\PP^n}(e_i))$, $\lambda_i\in \CC$ are such that $\sum \lambda_i e_i=0$ and $\overline{e}=(e_0,\dots,e_m)$ satisfies $\sum e_i=d$. The corresponding irreducible components of $\FF^1(\PP^n,d)$ are often denoted $\LL og(n,\overline{e})$.

Observe that a generic element $\omega\in\LL og(n,\overline{e})$ is the pullback under corresponding map $F:\PP^n\dashrightarrow \PP^m(\overline{e}/k)$ of degree $k=gcd(\overline{e})$ of the foliation induced by the element
\begin{equation} \label{log2} \alpha=\sum_{i=0}^m \lambda_i \widehat{z_i} dz_i \in \FF^1(\PP^m(\overline{e}/k), d/k).
\end{equation}
Straightforward calculation shows that if $\alpha$ is generic then $d\alpha$ vanishes in codimension $3$. Moreover, its tangent sheaf satisfies $\T_{\alpha}\simeq \g\otimes \OO_{\PP^m(\overline{e}/k)}$ for an abelian subalgebra $\g\subseteq H^0(\PP^m(\overline{e}/k),\T_{\PP^m(\overline{e}/k)})$ generated by elements of the form
\[ \X_r=\sum_{j=0}^m \lambda_{rj} z_j\frac{\partial}{\partial z_j} \]
for some $\lambda_{rj}\in \CC$, see for instance \cite[Section 2.2]{CD06}. In order to apply \cref{elCor}, we are left to verify that these points are generically reduced in $\FF^1(\PP^m(\ov{e}/k,d/k)$. Let us first consider the case where $\ov{e}=(1,\dots,1)$, so that $H^0(\PP^m,\T_{\PP^m})\simeq \mathfrak{gl}_{m+1}/\langle id\rangle$. Under this isomorphism, the elements $\X_r$ correspond to classes of diagonal matrices with coefficients $\lambda_{rj}$. By a classical result on Lie algebras, for a generic $\g$ as above every deformation $\g_\varepsilon$  will also be (after an appropriate change of coordinates) of this form. This is, $\g_\varepsilon$ is in the $Aut(\PP^m)$-orbit of some deformation induced by a perturbation of the $\lambda_{jr}$ or equivalently $\phi_\varepsilon\cdot\g_\varepsilon$ is generated by elements of the form
\[  (\X_j)_\varepsilon=\sum_{r=0}^m (\lambda_{jr})_\varepsilon z_r\frac{\partial}{\partial z_r}  \]
for some $\phi_\varepsilon\in Aut(\PP^m\times \Sigma)$.
In particular, by \cref{teoCP} above the closure of the union of the $Aut(\PP^n)$-orbits of the elements of the form \cref{log2} define a generically reduced irreducible component.

Now suppose that $\ov{e}$ is such that $e_{i+1}=e_i$ or $e_i< e_{i+1}\notin \langle e_0,\dots,e_i\rangle_\NN$.
In this case, one can check that $H^0(\PP^m(\ov{e}/k),\T_{\PP^m(\ov{e}/k)})$ is isomorphic to
$$\bigoplus_{l=1}^s \mathfrak{gl}_{m_l} /\langle e_1 id_1+\cdots+e_s id_s\rangle\subseteq \mathfrak{gl}_{m+1}/\langle e_1 id_1+\cdots+e_s id_s\rangle,$$
where the decomposition is induced by the clusters such that $e_i=e_j$. Arguing as above we can also conclude that in this situation the closure of the union of the $Aut(\PP^m(\ov{e}))$-orbits of the elements as in \cref{log2} define a generically reduced irreducible component of $\FF^1(\PP^m(\ov{e}/k),d/k)$. We are now in position to apply \cref{elCor} in order to obtain:
\begin{corollary}\label{logcorolario} Let $n\geq m+2$ and $\ov{e}=(e_0,\dots,e_m)$ such that either $e_{i+1}=e_i$ or $e_i<e_{i+1}\notin \langle e_0,\dots,e_i\rangle_\NN$. Then the variety
\[
\LL og(n,\ov{e})=\ov{\left\{ \omega=\sum_{i=0}^m \lambda_i \widehat{F_i} dF_i \hspace{0.2cm} : \hspace{0.2cm} \deg(F_i)=e_i \mbox{ and } \sum \lambda_i e_i=0  \right\} }
\]
is an irreducible component of $\FF^1(\PP^n,\sum_{i=0}^m e_i)$.
\end{corollary}
\end{example}

We will now consider pullbacks of foliations which are induced by rigid subalgebras $\g\subseteq H^0(\PP^n,\T_{\PP^n})$. In this case, using \cref{teoCP} we can reformulate \cref{elCor} as follows:

\begin{corollary} \label{corLie} Let $n\geq m+2$ and $\g\subseteq H^0(\PP^m,\T_{\PP^m})$ a rigid subalgebra of dimension $m-1$ such that $\exp(\g)$ acts with trivial stabilizers in codimension one. Suppose further that $d\omega(\g)$ vanishes in codimension greater than $2$. Then the variety $PB(n,k,\g)\subseteq \FF^1(\PP^n,k(m+1))$ whose generic element is a pullback of $\omega(\g)$ under a rational map $\PP^n\dashrightarrow \PP^m$ of degree $k$ is an irreducible component of $\FF^1(\PP^n,k(m+1))$.
\end{corollary}

\begin{example}\label{expgeneral} In the case of $\PP^3$, its Lie algebra of global vector fields admits a unique (up to inner automorphisms) embedding $\aff(\CC)\hookrightarrow H^0(\PP^3,\T_{\PP^3})$, where $\aff(\CC)$ is the Lie algebra of affine transformations $\aff(\CC)=\langle X,Y :  [X,Y]=Y \rangle$. This subalgebra satisfies the hypotheses of \cref{teoCP} and therefore defines a rigid irreducible component of $\FF^1(\PP^3,4)$. For $n\geq 3$, its pullback under linear projections $\PP^n\dashrightarrow \PP^3$ are known as the \textit{exceptional components} $\EE(n)$ of the spaces of foliations $\FF^1(\PP^n,4)$. These were first introduced in \cite{celn}. Applying \cref{corLie} to the case $\g=\aff(\CC)\subseteq H^0(\PP^3,\T_{\PP^3})$, we see that for $n\geq 5$ these stable families are included in a bigger family of irreducible components, namely
\[ \EE(n,k)=PB(n,k,\g) \subseteq \FF^1(\PP^n,4k).\]
\end{example}

\begin{example}\label{pullbackgm} More generally, consider the subalgebra $\g(m)\subseteq H^0(\PP^m, \T_{\PP^m})$ of dimension $m-1$ generated by the vector fields
\begin{align*} \X&=\sum_{j=0}^m (m-2i)z_j\frac{\partial}{\partial z_j},\\
Y_r&=\sum_{j=0}^{m-r} z_{j+r}\frac{\partial}{\partial z_j}, \hspace{0.3cm}1\leq r\leq m-1
\end{align*}
with Lie brackets $[\X,Y_r]=-2Y_r$ and $[Y_l,Y_r]=0$. In \cite[Section 6.3]{fj} it is shown that this subalgebra is rigid and its associated foliation satisfies the hypotheses of \cref{teoCP}. In particular, we can apply \cref{corLie} in order to deduce that for $n\geq m+2$ the variety $PB(n,k,\g(m))$ is an irreducible component of $\FF^1(\PP^n,k(m+1))$.
\end{example}

\begin{example}\label{pullbackg67} In \cite[Proposition 6.7]{fj} and its preceding discussion the authors construct two rigid subalgebras $\g_6=\langle \X,Y_1,\dots,Y_4 \rangle \subseteq H^0(\PP^6,\T_{\PP^6})$ and $\g_7=\langle \X',Y_1'\dots,Y_5'\rangle\subseteq H^0(\PP^7,\T_{\PP^7})$ satisfying the hypotheses of \cref{teoCP}. Its Lie algebra structure is given by
\[
[\X,Y_r]=-2rY_r, \hspace{0.2cm} [Y_1,Y_r]=Y_{r+1}
\]
and
\[
[\X',Y'_r]=-2rY'_r, \hspace{0.2cm} [Y'_1,Y'_r]=Y'_{r+1},\hspace{0.2cm} [Y'_2,Y'_3]=-\frac{5}{2}Y'_5
\]
respectively.
Again, \cref{corLie} implies that the varieties
\begin{align*} PB(n,k,\g_6)\subseteq \FF^1(\PP^n,7k),&\hspace{0.1cm} n\geq 8, \mbox{ and} \\
PB(n,k,\g_7)\subseteq \FF^1(\PP^n,8k), &\hspace{0.1cm} n\geq 9
\end{align*}
are irreducible components of the corresponding spaces of foliations.
\end{example}

\section{A word on pullbacks from toric varieties}

One of the inquiries that motivated this work is wether there exist stable families that arise as pullbacks from varieties $Y$ different than (weighted) projective spaces. In order to address this problem, the authors first considered the case where $Y$ is a normal toric variety. In this short digression we will see that most of the stable families that could arise from such situation are included in the irreducible components presented in this paper.

Let $Y=Y_\Sigma$ the simplicial toric variety of dimension $d$ associated to the fan $\Sigma$ in $\mathbb{R}^d$ and let $\{ v_0=(v_1^0,\dots,v_d^0),\dots,v_m=(v_1^m,\dots,v_d^m)\}$ be the elements in $\mathbb{Z}^d$ generating its one-dimensional cones. The class group $Cl(Y)$ is generated by the $m+1$ torus invariant irreducible divisors. The main result in \cite[Theorem 2.1]{COX1} establishes that $Y$ admits a description as a quotient
\[ Y \simeq \left( \CC^{m+1}\setminus Z\right) /G,\]
where $Z$ is a union of linear varieties of codimension greater than one and $G=Hom(Cl(Y),\CC^*)$. The action is just the restriction of the diagonal action $(\CC^*)^{m+1}\curvearrowright\CC^{m+1}$ after viewing $G$ as a subgroup of $(\CC^*)^{m+1}$ under the identification
\[
G=\{ (g_0,\dots,g_{m})\in (\CC^*)^{m+1} : \prod_{i=0}^m g_i^{v_j^i}=1 \hspace{0.2cm} \forall 1\leq j \leq d\}\ ,
\]
see \cite[Lemma 5.1.1]{CLS}. In particular, an element $(t^{a_0},\dots,t^{a_{m}})\in (\CC^*)^{m+1}$ lies in $G$ if and only if $a_0v_0+\cdots+a_{m}v_{m}=0$. Let us denote by $\pi_Y:\CC^{m+1}\setminus Z\to Y$ the quotient morphism.

It follows from \cite[Proposition 2.12]{gmv} that a general rational map onto a toric variety $F: \PP^n\dashrightarrow Y$ admits a unique (up to multiplication of elements in $G$) complete polynomial lifting. This is, a map (which will be denoted in the same manner) $F:\CC^{n+1}\setminus\{ 0\} \dashrightarrow  \CC^{m+1}\setminus Z$ making the diagram
\[
\xymatrix{
	\CC^{n+1}-\{0\} \ar[d]^\pi \ar@{-->}[r]^F & \CC^{m+1}-Z\ar[d]^{\pi_Y}\\
	\PP^n \ar@{-->}[r]^{F} &Y}
\]
commutative and such that the open set $Reg(F)\subseteq \PP^n$ where $F$ is defined is exactly
\[
Reg(F) = \PP^n\setminus \pi\left(F^{-1}(Z)\right).
\]
Let us denote by $e_i=\deg(F_i)$ the degrees of the homogeneous polynomials $F_i$ inducing this map and set $k=gcd(\ov{e})$. Observe that by \cite[Remark 2.14]{gmv} the degree $\bar{e}=(e_0,\dots,e_m)$ of the map $F$ is invariant under small perturbations and satisfies $e_0 v_0+\cdots + e_N v_N=0$, \emph{i.e.}, the one-parameter subgroup $\{(t^{e_0/k},\dots,t^{e_0/k}):t\in \CC^* \}\subseteq (\CC^*)^{m+1}$ is contained in $G$.
It follows that the identity $id:\CC^{m+1}\setminus \{0\}  \dashrightarrow \CC^{m+1}\setminus Z$ induces a surjective rational map
\[
\phi_{\ov{e}}: \PP^{m}(\bar{e}/k)  \dashrightarrow Y .
\]
 By construction, the map $F$ factorizes through $\phi_{\ov{e}}$, \emph{i.e.}, the degree $k$ rational map $F_{\ov{e}}: \PP^n \dashrightarrow \PP^{m}(\bar{e}/k)$ induced by the $F_i$'s fits in the commutative diagram
\[
\xymatrix{
	\PP^n \ar@{-->}[d]^{F_e} \ar@{-->}[r]^F & Y\\
	 \PP^m(\bar{e}/k) \ar@{-->}[ru]_{\phi_{\ov{e}}} &}
\]
Let us now fix an isomorphism $Cl(Y)\simeq \ZZ^{m+1-k}\times H$ for a finite abelian group $H$. By duality, this induces an isomorphism
\[
(\CC^*)^{m+1-k}\times Hom(H,\CC^*)\simeq G\subseteq (\CC^*)^{m+1}.
\]
Without loss of generality, we can assume that the one parameter subgroup generated by the first $\CC^*$-factor is of the form $t\mapsto (t^{e_0/k},\dots,t^{e_m/k})$. This is convenient in the sense that the pullback $Cl(Y)\to Cl(\PP^{m}(\bar{e}/k))$ corresponds to the projection onto the first coordinate $\ZZ^{m+1-d}\to \ZZ$. Let us denote by $Y_r$ the smooth locus of $Y$.

\begin{definition} Let $\mathcal{G}$ be a sheaf on $Y$ such that $\mathcal{G}\vert_{Y_r}\simeq \bigoplus_{i=1}^r \OO_Y(\alpha_i)$ for some $\alpha_i=(\alpha^i_1,\dots,\alpha^i_{m+1-d},h)\in Cl(Y)$. We say that $\mathcal{G}$ is non-positive with respect to $e$ if $\alpha^i_1\leq 0$ for every $1\leq i \leq r$.
\end{definition}

 With respect to the pullbacks of foliations under $\phi_{\ov{e}}$ we have the following lemma, whose proof is straightforward.

\begin{lemma} Let $\alpha$ be a foliation on $Y$ with split tangent sheaf $\T_{\alpha}$. The following are equivalent:
\begin{enumerate}
	\item $\mathcal{T}_{\alpha}$ is non-positive with respect to $e$ and $d\pi_Y^*\alpha$ vanishes in codimension greater than $2$.
	\item $\phi_{\ov{e}}^*\alpha$ has non-positive splitting and $d\phi_{\ov{e}}^*\alpha$ vanishes in codimension greater than $2$.
\end{enumerate}
\end{lemma}

As a consequence, it is straightforward that the possible stable families of foliations on $\PP^n$ whose generic element is a pullback of a foliation on $Y$ satisfying the above hypotheses are included in the ones presented in this work. This is,

\begin{proposition} Let $\mathcal{C}\subseteq \FF^1(\PP^n,\ell)$ be an irreducible component whose generic element is a pullback under a rational map of degree $\ov{e}$ of a split foliation on a toric variety $Y$ such that $\mathcal{T}_{\alpha}$ is non-positive with respect to $\ov{e}$ and $d\pi_Y^*\alpha$ vanishes in codimension greater than $2$.
If $n+2$ is greater than the number of irreducible torus-invariant divisors on $Y$, then $\mathcal{C}$ is one of the irreducible components described in \cref{corLie}.
\end{proposition}

\section*{Declarations}
\subsection*{Conflict of interests}
On behalf of all authors, the corresponding author states that there is no conflict of interest.


\subsection*{Funding}
The first author was supported by Instituto Nacional de Matemática Pura e Aplicada and Conselho Nacional de Desenvolvimento Científico e Tecnológico, Brazil. This author also acknowledges support from CNPq Projeto Universal 402936/2021-3: Geometria global de conexões, webs e folheacões holomorfas.  The third author was fully supported by Consejo Nacional de Investigaciones Científicas y Técnicas, Argentina. The fourth author was supported by Conselho Nacional de Desenvolvimento Científico e Tecnológico, Brazil, and Consejo Nacional de Investigaciones Científicas y Técnicas, Argentina.

\printbibliography

@article{COXBATYREV,
author = {V. V. Batyrev and D. A. Cox},
title = {{On the Hodge structure of projective hypersurfaces in toric varieties}},
volume = {75},
journal = {Duke Mathematical Journal},
number = {2},
publisher = {Duke University Press},
pages = {293 -- 338},
year = {1994},
doi = {10.1215/S0012-7094-94-07509-1},
URL = {https://doi.org/10.1215/S0012-7094-94-07509-1}
}

@book {DOLGACHEV,
AUTHOR = {Dolgachev, I.},
TITLE = {Weighted projective varieties},
BOOKTITLE = {Group actions and vector fields ({V}ancouver, {B}.{C}., 1981)},
SERIES = {Lecture Notes in Math.},
VOLUME = {956},
PAGES = {34--71},
PUBLISHER = {Springer, Berlin},
YEAR = {1982},
MRCLASS = {14L32 (14A05 14B05)},
MRNUMBER = {704986},
DOI = {10.1007/BFb0101508},
URL = {http://dx.doi.org/10.1007/BFb0101508},
}

@article {fj,
AUTHOR = {Cukierman, F. and Pereira, J. V.},
TITLE = {Stability of holomorphic foliations with split tangent sheaf},
JOURNAL = {Amer. J. Math.},
FJOURNAL = {American Journal of Mathematics},
VOLUME = {130},
YEAR = {2008},
NUMBER = {2},
PAGES = {413--439},
ISSN = {0002-9327},
CODEN = {AJMAAN},
MRCLASS = {32S65},
MRNUMBER = {2405162 (2009e:32034)},
MRREVIEWER = {M. G. Soares},
DOI = {10.1353/ajm.2008.0011},
URL = {http://dx.doi.org/10.1353/ajm.2008.0011},
}

@article {CLE,
AUTHOR = {Cerveau, D. and Lins Neto, A. and Edixhoven, S. J.},
TITLE = {Pull-back components of the space of holomorphic foliations on {${\mathbb C}{\mathbb P}(n)$}, {$n\geq 3$}},
JOURNAL = {J. Algebraic Geom.},
FJOURNAL = {Journal of Algebraic Geometry},
VOLUME = {10},
YEAR = {2001},
NUMBER = {4},
PAGES = {695--711},
ISSN = {1056-3911},
MRCLASS = {32S65},
MRNUMBER = {1838975 (2002h:32025)},
MRREVIEWER = {M. G. Soares},
}

@article {celn,
AUTHOR = {Cerveau, D. and Lins Neto, A.},
TITLE = {Irreducible components of the space of holomorphic foliations of degree two in {$\mathbf C{\rm P}(n)$}, {$n\geq 3$}},
JOURNAL = {Ann. of Math. (2)},
FJOURNAL = {Annals of Mathematics. Second Series},
VOLUME = {143},
YEAR = {1996},
NUMBER = {3},
PAGES = {577--612},
ISSN = {0003-486X},
CODEN = {ANMAAH},
MRCLASS = {32L30 (32G99)},
MRNUMBER = {1394970 (97i:32039)},
MRREVIEWER = {Alain H{\'e}naut},
DOI = {10.2307/2118537},
URL = {http://dx.doi.org/10.2307/2118537},
}

@article {omegar,
AUTHOR = {Calvo-Andrade, O.},
TITLE = {Irreducible components of the space of holomorphic foliations},
JOURNAL = {Math. Ann.},
FJOURNAL = {Mathematische Annalen},
VOLUME = {299},
YEAR = {1994},
NUMBER = {4},
PAGES = {751--767},
ISSN = {0025-5831},
CODEN = {MAANA},
MRCLASS = {32L30},
MRNUMBER = {1286897 (95i:32039)},
MRREVIEWER = {Dominique Cerveau},
DOI = {10.1007/BF01459811},
URL = {http://dx.doi.org/10.1007/BF01459811},
}

@article {saito,
AUTHOR = {Saito, K.},
TITLE = {On a generalization of de-{R}ham lemma},
JOURNAL = {Ann. Inst. Fourier (Grenoble)},
FJOURNAL = {Universit\'{e} de Grenoble. Annales de l'Institut Fourier},
VOLUME = {26},
YEAR = {1976},
NUMBER = {2},
PAGES = {vii, 165--170},
ISSN = {0373-0956},
MRCLASS = {58A10},
MRNUMBER = {0413155},
MRREVIEWER = {D. B. Fuks},
URL = {http://www.numdam.org/item?id=AIF_1976__26_2_165_0},
}

@Book{jou,
Author = {J. P. {Jouanolou}},
Title = {{Equations de Pfaff alg\'ebriques.}},
FJournal = {{Lecture Notes in Mathematics}},
Journal = {{Lect. Notes Math.}},
ISSN = {0075-8434; 1617-9692/e},
Volume = {708},
Year = {1979},
Publisher = {Springer, Cham},
Language = {French},
MSC2010 = {58A17 14D05 57R30 32Sxx 32C30 14F10 34M99 37C85},
Zbl = {0477.58002}
}

@book {CLS,
AUTHOR = {Cox, D. A. and Little, J. B. and Schenck, H. K.},
TITLE = {Toric varieties},
SERIES = {Graduate Studies in Mathematics},
VOLUME = {124},
PUBLISHER = {American Mathematical Society, Providence, RI},
YEAR = {2011},
PAGES = {xxiv+841},
ISBN = {978-0-8218-4819-7},
MRCLASS = {14M25 (05A15 05E45 52B12)},
MRNUMBER = {2810322},
MRREVIEWER = {Ivan Arzhantsev},
DOI = {10.1090/gsm/124},
URL = {https://doi.org/10.1090/gsm/124},
}

@article {COX1,
AUTHOR = {Cox, D. A.},
TITLE = {The homogeneous coordinate ring of a toric variety},
JOURNAL = {J. Algebraic Geom.},
FJOURNAL = {Journal of Algebraic Geometry},
VOLUME = {4},
YEAR = {1995},
NUMBER = {1},
PAGES = {17--50},
ISSN = {1056-3911},
MRCLASS = {14M25},
MRNUMBER = {1299003},
MRREVIEWER = {Mina Teicher},
}

@article{moli,
AUTHOR = {Molinuevo, A.},
TITLE = {Unfoldings and deformations of rational and logarithmic foliations},
JOURNAL = {Ann. Inst. Fourier (Grenoble)},
FJOURNAL = {Universit\'e de Grenoble. Annales de l'Institut Fourier},
VOLUME = {66},
YEAR = {2016},
NUMBER = {4},
PAGES = {1583--1613},
ISSN = {0373-0956},
MRCLASS = {37F75 (14D20 32S65)},
MRNUMBER = {3494179},
URL = {http://aif.cedram.org/item?id=AIF_2016__66_4_1583_0},
}

@article{G,
AUTHOR = {Grothendieck, A.},
TITLE  = {{\'El\'ements de g\'eom\'etrie alg\'ebrique. IV: \'Etude locale des sch\'emas et des morphismes de sch\'emas. (Troisi\`eme partie). R\'edig\'e avec la colloboration de Jean Dieudonn\'e.}},
FJOURNAL = {{Publications Math\'ematiques}},
JOURNAL = {{Publ. Math., Inst. Hautes \'Etud. Sci.}},
ISSN = {0073-8301; 1618-1913/e},
VOLUME = {28},
PAGES = {1--255},
YEAR = {1966},
PUBLISHER = {Springer, Berlin/Heidelberg; Institut des Hautes \'Etudes Scientifiques, Bures-sur-Yvette},
LANGUAGE = {French},
DOI = {10.1007/BF02684343},
MSC2010 = {14-02 14Axx},
ZBL = {0144.19904},
}

@book {M,
AUTHOR = {Mumford, D.},
TITLE = {The red book of varieties and schemes},
SERIES = {Lecture Notes in Mathematics},
VOLUME = {1358},
EDITION = {expanded},
NOTE = {Includes the Michigan lectures (1974) on curves and their Jacobians, With contributions by Enrico Arbarello},
PUBLISHER = {Springer-Verlag, Berlin},
YEAR = {1999},
PAGES = {x+306},
ISBN = {3-540-63293-X},
MRCLASS = {14-01 (14A15 14H40 14H42)},
MRNUMBER = {1748380},
MRREVIEWER = {Arnaud Beauville},
DOI = {10.1007/b62130},
URL = {https://doi.org/10.1007/b62130},
}

@book {S,
    AUTHOR = {Stanley, R. P.},
     TITLE = {Combinatorics and commutative algebra},
    SERIES = {Progress in Mathematics},
    VOLUME = {41},
   EDITION = {Second},
 PUBLISHER = {Birkh\"{a}user Boston, Inc., Boston, MA},
      YEAR = {1996},
     PAGES = {x+164},
      ISBN = {0-8176-3836-9},
   MRCLASS = {05-02 (13F50)},
  MRNUMBER = {1453579},
MRREVIEWER = {Volkmar Welker},
}

@article {V,
    AUTHOR = {Velazquez, S.},
     TITLE = {Toric foliations with split tangent sheaf},
  JOURNAL = {Bull. Sci. Math.},
  FJOURNAL = {Bulletin des Sciences Math\'{e}matiques},
    VOLUME = {175},
      YEAR = {2022},
     PAGES = {Paper No. 103099, 23},
      ISSN = {0007-4497},
   MRCLASS = {32S65 (14M25 37F75)},
  MRNUMBER = {4370487},
       DOI = {10.1016/j.bulsci.2022.103099},
       URL = {https://doi.org/10.1016/j.bulsci.2022.103099},
}

@article {GMV,
    AUTHOR = {Gargiulo Acea, J. and Molinuevo, A. and Velazquez, S.},
     TITLE = {Rational pullbacks of toric foliations},
   JOURNAL = {Forum Math.},
  FJOURNAL = {Forum Mathematicum},
    VOLUME = {35},
      YEAR = {2023},
    NUMBER = {3},
     PAGES = {843--861},
      ISSN = {0933-7741,1435-5337},
   MRCLASS = {14D20 (14B10 14M25 32S65 37F75)},
  MRNUMBER = {4587288},
MRREVIEWER = {Claudia\ R.\ Alc\'antara},
       DOI = {10.1515/forum-2022-0265},
       URL = {https://doi.org/10.1515/forum-2022-0265},
}

@book{Artin76,
  title={Lectures on deformations of singularities},
  author={Artin, M. and Seshadri, C. S. and Tannenbaum, A.},
  volume={54},
  year={1976},
  publisher={Tata Institute of Fundamental Research Bombay}
}

@article {CD06,
    AUTHOR = {D\'{e}serti, J. and Cerveau, D.},
     TITLE = {Feuilletages et actions de groupes sur les espaces projectifs},
   JOURNAL = {M\'{e}m. Soc. Math. Fr. (N.S.)},
  FJOURNAL = {M\'{e}moires de la Soci\'{e}t\'{e} Math\'{e}matique de France. Nouvelle S\'{e}rie},
    NUMBER = {103},
      YEAR = {2005},
     PAGES = {vi+124 pp. (2006)},
      ISSN = {0249-633X},
   MRCLASS = {37F75 (32M17 32M25 32S65)},
  MRNUMBER = {2200857},
MRREVIEWER = {Jorge Vit\'{o}rio Pereira},
       DOI = {10.24033/msmf.415},
       URL = {https://doi.org/10.24033/msmf.415},
}

@book {MATSU,
    AUTHOR = {Matsumura, H.},
     TITLE = {Commutative ring theory},
    SERIES = {Cambridge Studies in Advanced Mathematics},
    VOLUME = {8},
   EDITION = {Second},
      NOTE = {Translated from the Japanese by M. Reid},
 PUBLISHER = {Cambridge University Press, Cambridge},
      YEAR = {1989},
     PAGES = {xiv+320},
      ISBN = {0-521-36764-6},
   MRCLASS = {13-01},
  MRNUMBER = {1011461},
}

@article {degree3LPdC,
    AUTHOR = {da Costa, R. C. and Lizarbe, R. and Pereira, J. V.},
     TITLE = {Codimension one foliations of degree three on projective spaces},
   JOURNAL = {Bull. Sci. Math.},
  FJOURNAL = {Bulletin des Sciences Math\'ematiques},
    VOLUME = {174},
      YEAR = {2022},
     PAGES = {Paper No. 103092, 39},
      ISSN = {0007-4497,1952-4773},
   MRCLASS = {37F75 (32G13)},
  MRNUMBER = {4354288},
MRREVIEWER = {F.\ Touzet},
       DOI = {10.1016/j.bulsci.2021.103092},
       URL = {https://doi.org/10.1016/j.bulsci.2021.103092},
}


\

\

\

{\tiny
\noindent
\begin{tabular}{l l}
	Javier Gargiulo Acea$^*$ \hspace{2cm}\null&\textsf{jngargiulo@gmail.com. ORCID 0000-0002-9675-1479.}\\
	Ariel Molinuevo$^\dag$  &\textsf{arielmolinuevo@gmail.com. ORCID 0000-0003-0286-1237.}\\
Federico Quallbrunn$^\ddag$ &\textsf{fquallb@dm.uba.ar.  ORCID 0000-0001-7856-3135.}\\
	Sebasti\'an Velazquez$^\S$  &\textsf{sebastian.velazquez@kcl.ac.uk. ORCID 0000-0001-8199-2600. }\\
\end{tabular}}

\

\

\

{\tiny
\noindent
\begin{tabular}{l l}
	$^*${Departamento de Matemática Aplicada} & $^\dag$Instituto de Matemática\\
	{Rua Prof. Marcos Waldemar de Freitas Reis s/n} &Av. Athos da Silveira Ramos 149 \\
	{Bloco G, Campus do Gragoatá} &Bloco C, Centro de Tecnologia, UFRJ  \\
	{S\~ao Domingos} &Cidade Universitária, Ilha do Fund\~ao \\
	{CEP 24210-201} &CEP 21941-909  \\
	{Niteroi, RJ}  &Rio de Janeiro, RJ \\
	Brasil &Brasil  \\
	&   \\
	&   \\
	$^\ddag$ {Departamento de Matem\'atica}  & \S{King's College London}  \\
	{Pabell\'on I, Ciudad Universitaria}  & {Strand Building} \\
	{FCEyN, UBA}  & {WC2R 2LS}\\
	{CP C1428EGA}  & {London }  \\
	{Buenos Aires}   & {United Kingdom} \\
	{Argentina}    & {}
\end{tabular}}

\end{document}